\documentclass[11pt]{amsart}
\usepackage{amsmath,amsfonts, mathtools,amssymb,latexsym,amssymb,amsthm,adjustbox,mathrsfs,amsbsy, bm,tikz-cd,indentfirst, makecell,graphicx, verbatim, calc, enumerate,xy}
\usepackage[top=1in,bottom=1.3in,left=1.1in,right=1.1in]{geometry}
               
\usepackage[colorlinks=true,linkcolor=blue]{hyperref}

\theoremstyle{plain}
\newtheorem{theorem}{Theorem}[section]
\newtheorem{lemma}[theorem]{Lemma}

\newtheorem{corollary}[theorem]{Corollary}
\newtheorem{proposition}[theorem]{Proposition}

\newtheorem{question}[theorem]{Question}

\newtheorem{example}[theorem]{Example}
\newtheorem{point}[theorem]{}

\newcommand{\m}{\mathfrak{m}}

\newcommand{\R}{\mathcal{R}}

\newcommand{\Z}{\mathbb{Z}}
\newcommand{\N}{\mathbb{N}}
\newcommand{\wt}{\widetilde}
\newcommand{\n}{\mathfrak{n}}

\newcommand{\cx}{\operatorname{cx}}

\newcommand{\Hom}{\operatorname{Hom}}

\newcommand{\Ext}{\operatorname{Ext}}
\newcommand{\Tor}{\operatorname{Tor}}

\usepackage{graphicx}
\usepackage{subfigure}
\usepackage{epsfig}
\usepackage{epstopdf}
\usepackage{float} 
\theoremstyle{plain}

\usepackage{xcolor}
\usepackage{titlesec}

\titleformat{name=\section}{}{\thetitle.}{0.5em}{\centering\scshape\large}

\begin{document}
\title{\large \textbf{Asymptotic behavior of invariants of syzygies of maximal Cohen-Macaulay modules }}
\author{\textsc{Tony J. Puthenpurakal}}
\email{tputhen@math.iitb.ac.in}
\address{Department of Mathematics, IIT Bombay, Powai, Mumbai 400 076, India}
\author{\textsc{Samarendra Sahoo}}
\email{204093008@iitb.ac.in}

\address{Department of Mathematics, IIT Bombay, Powai, Mumbai 400 076, India}
\date{\today}

\subjclass{Primary 13A30, 13C14, 13D40, 13D02, 13D07, 13D45}
\keywords{Complete intersection rings, Complexity of modules, Associated graded rings, Castelnuovo-Mumford regularity, Ulrich modules, superficial element, Hilbert coefficients}

\begin{abstract} Let $(A,\mathfrak{m})$ be a complete intersection ring of codimension $c\geq 2$ and dimension $d\geq 1$. Let $M$ be a finitely generated maximal Cohen-Macaulay $A$-module and  $M_i=\text{Syz}^A_{i}(M)$ be the $i$-th syzygy of $M$. Let $e^{I}_i(M)$ be the $i$-th Hilbert coefficient of $M$ with respect to an $\mathfrak{m}$-primary ideal $I$. In \cite{Quasipolynomial}, it is proved that the function $i\mapsto e^I_1(M)$ is of quasi-polynomial type with period $2$ and degree $\leq \text{cx}_A(M)-1$, where $\text{cx}_A(M)$ is the complexity of $M.$ We prove that for $I=\mathfrak{m}$, the associated polynomials have the same degree equal to $\text{cx}_A(M)-1$. For $\text{cx}_A(M)=2$ and $j=0,1$, we prove  $$\lim_{i\to \infty}\dfrac{e^{\mathfrak{m}}_1(M_{2i+j})}{i}\geq \lim_{i\to \infty}\dfrac{e^{\mathfrak{m}}_0(M_{2i+j})}{i}-\lim_{i\to \infty}\dfrac{\mu(M_{2i+j})}{i}.$$ When equality holds, we prove that the set $\{\text{reg }G_{\mathfrak{m}}(M_i)\,|\, i\geq 0\}$ is bounded, where reg $G_{\mathfrak{m}}(M_i)$ is the Castelnuovo-Mumford regularity of $G_{\mathfrak{m}}(M_i).$
    
\end{abstract}

\maketitle

\section{Introduction}
Let $(A,\mathfrak{m})$ be a Noetherian local ring with residue field $k=A/\mathfrak{m}$. Let $I$ be an $\mathfrak{m}$-primary ideal and $M$ be a finitely generated $A$-module. We denote the associated graded ring of $A$ with respect to the ideal $I$ by $G_I(A)=\bigoplus_{n\geq 0}I^n/I^{n+1}$, and the associated graded module of $M$ with respect to $I$ by $G_I(M)=\bigoplus_{n\geq 0}I^nM/I^{n+1}M$, considered as a graded $G_I(A)$-module. Let $\ell_A(E)$ denote the length of an $A$-module $E$.

Let $M$ be an $A$-module of dimension $r$. The Hilbert series of $M$ is given by the formula $$H_{M,I}(z)=\sum_{i\geq 0}\ell_A(I^iM/I^{i+1}M)z^i=\frac{h_{M,I}(z)}{(1-z)^r},$$ where $h_{M,I}(z)\in \Z[z]$ and is called the $h$-polynomial of $M$ with respect to $I$. The integer $e^I_i(M)=h^{(i)}_{M,I}(1)/i!$ is referred to as the $i$-th Hilbert coefficient of $M$ with respect to $I$, where $h^{(i)}_{M,I}(z)$ is the $i$-th derivative of $h_{M,I}(z)$. When $I=\mathfrak{m}$, we write $G_{\m}(A)=G(A)$, $G_{\m}(M)=G(M)$, $h_{M,\m}(z)=h_{M}(z)$, and $e^{\m}_i(M)=e_i(M)$.

 Recall a function $f:\Z_{\geq 0}\to \Z_{\geq 0}$ is said to be of quasi-polynomial type with period $2$ if there exist polynomials $P_0$ and $P_1$ such that $f(2m)=P_0(m)$ and $f(2m+1)=P_1(m)$ for all $m\gg0$. We define the degree of 
$f$ as $$\text{deg}(f)=\text{max}\{\text{deg}(P_0),\text{deg}(P_1)\}.$$ For convenience, set deg$(0)=-\infty$ (for a general definition, see \ref{qasi}). We refer to the polynomials $P_0$ and $P_1$ as the polynomials associated to $f$.

There has been extensive research on the behavior of the ``Betti numbers of the Cohen-Macaulay module over complete intersection rings''. For instance, see \cite{Avramov}, \cite{Peeva}, and \cite{Eisenbud}. Let $\beta_i(M)$ denote the $i$-th Betti number of $M$, and $\mu(M)$ denote the cardinality of the minimal set of generators of $M$. Let $M_i=\text{Syz}^A_i(M)$ for all $i\geq 1$, representing the $i$-th syzygy of $M$ in its minimal free resolution, and for convenience, set $M_0=M$. Note that $\beta_i(M)=\mu(M_i)$ for all $i\geq 0$. One interesting result due to Avramov (\cite{Avramov}, Theorem 4.1) is: let $(A,\m)$ be a complete intersection ring and $M$ a maximal Cohen-Macaulay $A$-module. Then the function $i\mapsto \beta_i(M)$ is of quasi-polynomial type with period $2$ and degree equal to $\text{cx}_A(M)-1$. In the same paper, he also established that the associated polynomials $P_0$ and $P_1$ have the same degree and leading coefficient. Similar questions can be posed regarding the behavior of other invariants of $M_i$ for sufficiently large $i.$

 In \cite{Quasipolynomial}, the author focused on the Hilbert coefficients of the syzygies of the modules and established that for a maximal Cohen-Macaulay module $M$ over a complete intersection ring (see \cite{Quasipolynomial}, 1.4), the functions $i\mapsto e^I_j(M_i)$ are of quasi-polynomial type with period 2 and degree at most $\text{cx}_A(M)-1$ for all $0\leq j\leq 2$. He also proved that the function $i\mapsto e^I_0(M_i)$ is of quasi-polynomial type with period $2$ and the associated polynomials $P_0$ and $P_1$ have the same degree equal to $\text{cx}_A(M)-1$. Here, the complexity of $M$ over $A$ ($\text{cx}_A(M)$) measures the rate of growth of its Betti numbers. For a detailed definition, see \ref{Thm:5.4}. 
 
In this paper, we explore the behavior of $e_1(M_i)$ and reg $G(M_i)$ (for the definition of regularity, see \ref{regL}) for sufficiently large values of $i$ over complete intersection rings. In the first part of this paper (Section $3$), we will address the following questions:

\begin{question}\label{que}
Let $(A,\mathfrak{m})$ be a complete intersection ring, $I$ an $\m$-primary ideal of $A$ and $M$ be a maximal Cohen-Macaulay $A$-module.
\begin{enumerate}
    \item[{(a)}] Is the function $i\mapsto e^I_j(M_i)$ of quasi-polynomial type with period $2$ and degree at most $\operatorname{cx}_A(M)-1$ for each $j\geq 0$ ?
     \item[{(b)}] If part $(a)$ has positive answer, do the associated polynomials $P_0$ and $P_1$ have the same degree equal to $\operatorname{cx}_A(M)-1$?

\end{enumerate}
\end{question}
Note that the first author answered part $(a)$ for $j=0,1,2$, and $(b)$ for $j=0$. In this paper, we are able to answer part $(b)$ for $j=1$ and $I=\m$ under certain conditions on $A$. For the equality of the leading coefficients of $P_0$ and $P_1$, we have both positive and negative answers for $j=0,1$ (see Example \ref{example1}). However, we are unable to provide sufficient conditions for the leading coefficients to be equal. Additionally, part $(a)$ and part $(b)$ in Question \ref{que} remain open for $j\geq 3$ and $j\geq 2$ respectively.

To answer Question \ref{que}(b) for the first Hilbert coefficient, additional conditions are needed for both the ring and the ideal. For example, let $I=(x_1,\ldots ,x_d)$ be an ideal of $A$ generated by a maximal $A$-regular sequence. Then $x_1,\ldots ,x_d$ is a maximal $M_i$-regular sequence for all $i\geq 0.$ Therefore, for all $i\geq 0$, the Hilbert series of $M_i$ is $H_{M_i,I}(z)=\dfrac{h_{M_i,I}(z)}{(1-z)^d},$ where $h_{M_i,I}(z)=\ell_A(M_i/IM_i)$ (see \cite{Bruns}, Theorem 1.1.8). Hence $e^I_1(M_i)=0$ for all $i\geq 0.$ This says that the function $i\mapsto e_1(M_i)$ is identically zero. It follows that the associated polynomials have degree equal to $-\infty$, which we aim to avoid. We answer Question \ref{que}(b) for the maximal ideal as follows:

\begin{theorem}[see Theorem \ref{e1}]
\label{2nd}
     Let $(A,\m)$ be a complete intersection ring of codimension $c\geq 2$ and dimension $d\geq 1$. Let $M$ be a maximal Cohen-Macaulay $A$-module. Set $M_i=\operatorname{Syz}^A_i(M)$ for all $i\geq 0.$ Then the function $i\mapsto e_1(M_i)$ is of quasi-polynomial type with period $2$ and the associated polynomials have the same degree equal to ${\cx}_A(M)-1$.
\end{theorem}

With the hypotheses of Theorem \ref{2nd}, the functions $i \mapsto \mu(M_i)$, $i \mapsto e_0(M_i)$, and $i \mapsto e_1(M_i)$ are all of quasi-polynomial type with period $2$, and in each case, the associated polynomials $P_0$ and $P_1$  have the same degree, namely $\text{cx}_A(M)$. Next, under the same hypotheses of Theorem \ref{2nd}, we compare the leading coefficients of the polynomials associated to these functions and prove for $j=0\text{ and }1$, 
\begin{equation}\label{1}
    \lim_{i\to \infty}\dfrac{e_1(M_{2i+j})}{i^{\operatorname{cx}_A(M)-1}}\geq \lim_{i\to \infty}\dfrac{e_0(M_{2i+j})}{i^{\operatorname{cx}_A(M)-1}}-\lim_{i\to \infty}\dfrac{\mu(M_{2i+j})}{i^{\operatorname{cx}_A(M)-1}}.
\end{equation}

Let $M$ be a maximal Cohen-Macaulay $A$-module of dimension $d$. One interesting question is how the invariants $a_j(G(M_i))=\text{max}\{n \, |\, H^j(G(M_i))_n\neq 0\}$, where $H^j(G(M_i)$ is the $j$-th graded local cohomology of $G(M_i)$ with respect to graded maximal ideal of $\R=\bigoplus_{n\geq 0}\m^nt^n$ (for a detailed definition see \ref{regL}), change for large values of $i$ and for $0\leq j\leq d$. In (\cite{Quasipolynomial}, Theorem 8.1), the first author established that the set $\{a_{d-1}(G(M_i))/i^{\operatorname{cx}_A(M)-1}\, |\, i\geq 1\}$ is bounded. In this paper, we prove the following:

 \begin{theorem}[see Theorem \ref{regMn}]
 \label{5th}
     Let $(A,\m)$ be a complete intersection ring of codimension $c\geq 2$ and dimension $d\geq 1$. Let $M$ be a maximal Cohen-Macaulay $A$-module. If $\operatorname{cx}_A(M)=2$ and $$\lim_{i\to \infty}\dfrac{e_1(M_{2i+j})}{i}= \lim_{i\to \infty}\dfrac{e_0(M_{2i+j})}{i}-\lim_{i\to \infty}\dfrac{\mu(M_{2i+j})}{i}$$ for $j=0,1$. Then the set $\{\operatorname{reg }G(M_i)\,|\, i\geq 0\}$ is bounded, where $\operatorname{reg} G(M_i)$ is the Castelnuovo-Mumford regularity of $G(M_i)$ as a graded $G(A)$-module.
\end{theorem}

We will now provide a brief overview of this paper. In Section 2, we will present some necessary preliminary results. In Section 3, we give the proof of Theorem \ref{2nd}. In Section 4, we prove the inequality in \eqref{1} and discuss the consequences when the equality holds. Finally, in Section 5, we prove Theorem \ref{5th}.
 
\section{Preliminaries}
Throughout this paper, all rings considered are Noetherian, and all modules considered, unless stated otherwise, are finitely generated.

\begin{point}\label{basechange}
Base change \\
\normalfont
   Let $\phi:(A,\mathfrak{m})\to (A',\mathfrak{m}')$ be a local ring homomorphism. Assume $A'$ is a faithfully flat $A$-algebra with $\mathfrak{m}A'=\mathfrak{m}'$. Let $M$ be an $A$-module. Set $M'=M\otimes A'$. In these cases, it can be shown that
\begin{enumerate}
    \item $\ell_A(M)=\ell_{A'}(M')$.
    \item $H_M(n)(=\ell_{A}(\m^nM/\m^{n+1}M))=H_{M'}(n)(=\ell_{A'}({\m'}^nM'/{\m'}^{n+1}M'))$ for all $n\geq 0$ and hence $H_M(z)=H_{M'}(z).$
    \item dim$_{A}M=$ dim$_{A'}M'$ and $\text{ depth}_A M= \text{ depth}_{A'} M'$.
    \item depth$_{G(A)}G(M)=\text{ depth}_{G(A')} G(M')$.
    \item $A$ is a local complete intersection iff $A'$ is a local complete intersection.
    \item $\beta^A_i(M)=\beta^{A'}_i(M')$. This implies $\operatorname{cx}_A(M)=\operatorname{cx}_{A'}(M').$
\end{enumerate}
\end{point}
The specific base changes we do are the following:
\begin{enumerate}
    \item[(i)] $A'=\hat{A}$ the completion of $A$ with respect to the maximal ideal.
    \item[(ii)] $A'=A[[X]]_S$, where $S=A[[X]]\setminus \m A[[X]].$ The maximal ideal of $A'$ is $\m '=\m A'$ and the residue field of $A'$ is $K=k((X)),$ which is uncountable.
\end{enumerate}

\begin{point}
Superficial element\\ \label{superficial}
    \normalfont
    Let $(A,\m)$ be a local ring and $M$ be an $A$-module. An element $x\in  \m$ is called $M$-superficial with respect to $\m$ if there exists $c\in \N$ such that for all $n \geq c$, $(\m^{n+1}M :_Mx)\cap \m^cM = \m^nM$. If depth$_AM>0$ then one can show that a $M$-superficial element is $M$-regular. Furthermore $(\m^{n+1}M :_Mx) = \m^nM$ for $n\gg0.$ Superficial elements exist if the residue field is infinite. See [\cite{JDSally}, p. 6-8] for proof of these facts when $M = A$. The same proof works in general.

    A sequence $x_1, \ldots ,x_r$  in $\m$ is said to be $M$-superficial sequence if $x_1$ is $M$-superficial and $x_i$ is $M/(x_1,\ldots ,x_{i-1})M$-superficial for $2\leq i \leq r .$
\end{point}

\begin{point}
The Rees module $L(M)$\\
\label{hpolyn}
      \normalfont
     Let $(A,\m)$ be a local ring and $M$ be an $A$-module. Define  $L(M)=\bigoplus_{n\geq 0}M/\m^{n+1}M$. Let $\R=A[\m t]$ be Rees ring and
  $\mathcal{R}(M)=\bigoplus_{n\geq 0}\m^nMt^n$ be Rees module of $M$. The Rees ring $\mathcal{R}$ is a subring of $A[t]$. So $A[t]$ is an $\mathcal{R}$-module. Therefore $M[t] = M\otimes_A A[t]$ is an $\R$-module. The exact sequence $$0\to \mathcal{R}(M)\to M[t]\to L(M)(-1)\to 0 $$ defines an $\mathcal{R}$-module structure on  $L(M)(-1)$ and hence on $L(M)$ (for more details see (\cite{Part1}, definition 4.2)). Note that $L(M)$ is not a finitely generated $\R$-module.
\end{point}

\begin{point}\label{samuel}
    \normalfont
 Let $(A,\m)$ be a local ring and $M$ be an $A$-module of dimension $r$. The Hilbert-Samuel function of the module $M$ with respect to the ideal $\mathfrak{m}$ is defined as the function $n \mapsto \ell_A(M/\mathfrak{m}^{n+1}M)$. It is well-known that for large values of $n$, this function can be expressed as a polynomial $P_M(n)$ of degree equal to the dimension of $M$, known as the Hilbert-Samuel polynomial of $M$. The polynomial can be written in the form $$P_M(n)=\ell(M/\m^{n+1}M)=e_0(M)\binom{n+r}{r}-e_1(M)\binom{n+r-1}{r-1}+\ldots +(-1)^re_r(M),$$ where $e_0(M), \ldots , e_r(M)$ are the Hilbert coefficients of $M$.

 It can be easily seen that the formal power series $\sum_{n\geq 0}\ell_A{(M/\m^{n+1}M)}z^n$ represents a rational function $$\sum_{n\geq 0}\ell_A{(M/\m^{n+1}M)}z^n=\dfrac{h_M(z)}{(1-z)^{r+1}}.$$ Here $h_M(z)$ is the $h$-polynomial of $M$ w.r.t. $\m$.  \end{point} 

\begin{point}
\label{L_1M}
    \normalfont
    Let $(A,\m)$ be a local ring and $M$ be an $A$-module. Set $L_1(M)=\bigoplus_{n\geq 0}\Tor^{A}_1(M,A/\m^{n+1}).$ Note that $L(A)(-1)\otimes M=L(M)(-1).$ By tensoring the exact sequence $0\to \mathcal{R}\to A[t]\to L(A)(-1)\to 0$ with $M$, we get $$0\to L_1(M)(-1)\to \mathcal{R}\otimes M\to M[t]\to L(M)(-1)\to 0.$$ Thus $ L_1(M)(-1)$ is a $\mathcal{R}$-submodule of $\mathcal{R}\otimes M.$  Since $\mathcal{R}\otimes M$ is finitely a generated $\mathcal{R}$-module, so $L_1(M)$ is a finitely generated $\mathcal{R}$-module.

    Let $0\to N\to P\xrightarrow{f} M\to 0$ be a short exact sequence of $A$-modules. After tensoring with $A/\m^n$, we get the following long exact sequence of homology $$\ldots \to \Tor^A_1(M, A/\m^n)\to N/\m^nN\to P/\m^nP\xrightarrow{f_n=f\otimes 1_n} M/\m^nM\to 0.$$ Taking direct sum over $n\geq 1$, we get the following exact sequence of $\R$-modules: $$\ldots \to L_1(M)\to L(N)\to L(P)\xrightarrow{L(f)} L(M)\to 0,$$ where $L(f)=\bigoplus_{n\geq 1}f_n$ is the induced map on direct sum.
\end{point}

\begin{point}[See \cite{HCCMM}]
\label{hcs}
\normalfont
Let $(A,\m)$ be a Cohen-Macaulay local ring and $M$ a finite Cohen-Macaulay $A$-module of dimension $r>0.$ Let $x_1,\ldots ,x_s$ be $A\oplus M$-superficial sequence with $s\leq r$. Set $J=(x_1,\ldots ,x_s)$. The local ring $(B,\mathfrak{p})=(A/J,\m/J)$ and $B$-module $N=M/JM$ satisfy:
\begin{enumerate}
    \item $x_1,\ldots ,x_s$ is a $A\oplus M$-regular sequence.
    \item dim $M$=dim $N+s$.
    \item $N$ is a Cohen-Macaulay $B$-module.
    \item $e_i(M)=e_i(N)$ for all $0\leq i\leq r-s.$
    
    \item $\mu(M)=\mu(N).$
    \item If $G(M)$ is Cohen-Macaulay then the sequence $x^*_1,\ldots ,x^*_s$ is $G(M)$-regular, where $x^*_i$ is the initial form of $x_i$ in $G(A).$ We also have the $h$-polynomial of $M$ and $N$ with respect to $\m$ are same.
   
    \item $e_1(M)\geq e_0(M)-\mu(M).$
    \item If $e_1(M) = e_0(M) - \mu(M)$, then $G(M)$  has minimal multiplicity and hence is Cohen-Macaulay. Here, $G(M)$ has minimal multiplicity means that the degree of the $h$-polynomial satisfies $h_{M}(z)\leq 1.$
\end{enumerate}
\end{point}

\begin{point}
Quasi-polynomial \label{qasi}\\
    \normalfont
    A function $f:\Z_{\geq 0}\to \Z_{\geq 0}$ is said to be of quasi-polynomial type with period $r>0$ if there exist polynomials $P_0,P_1,\ldots ,P_{r-1}$ such that $f(mr+i)=P_i(m)$ for all $m\gg0$ and $i=0,1,\ldots ,r-1.$
\end{point}

\begin{lemma}[\cite{Beck}, Lemma 9]
\label{quaso}
    Let $f:\Z_{\geq 0}\to \Z_{\geq 0}$. Then the following are equivalent:
    \begin{enumerate}
        \item $f$ is of quasi-polynomial type with period 2.
        \item $\sum_{n\geq 0}f(n)z^n=\dfrac{h(z)}{(1-z^2)^c}$ for some $h(z)\in \Z[z]$ and $c\geq 0.$
    \end{enumerate}
   Moreover if $P_0,P_1\in \mathbb{Q}[X]$ are polynomials such that $f(2m+i)=P_i(m)$ for all $m\gg0$ and $i=0,1$ then deg $P_i\leq c-1$. 
\end{lemma}

 Set deg $f=\text{max}\{\text{deg } P_0,\text{deg }P_1\}$.

\begin{corollary}
\label{deg}
     Let $f:\Z_{\geq 0}\to \Z_{\geq 0}$. If the function $g(n)=f(n)-f(n-2)$ is of quasi-polynomial type with period 2, then so is $f$. Furthermore  $\operatorname{deg }f=\operatorname{deg }g+1$.
\end{corollary}

\begin{proof}
    Set $u(z)=\sum_{n\geq 0}f(n)z^n$ and $v(z)=\sum_{n\geq 0}g(n)z^n$. Then by our hypothesis and Lemma \ref{quaso}, we get 
    $$(1-z^2)u(z)=\dfrac{h(z)}{(1-z^2)^c}=v(z)$$ for some $h(z)\in \Z[z]$ and $c\geq 0.$ By Lemma \ref{quaso}, it follows that  $f$ is of quasi-polynomial type with period 2. Furthermore, suppose $P_0(n)=g(2n)=f(2n)-f(2n-2)$ for all $n\geq n_0.$ This implies that $f(2n)=f(2n_0-2)+\sum_{i= n_0}^{n}P_0(i)=Q_0(n)$ for all $n\gg0$ and deg $Q_0=\text{deg }P_0+1$. Similarly, we can show that $f(2n+1)=f(2n_0-1)+\sum_{i= n_1}^{n}P_1(i)=Q_1(n)$ for all $n\gg0$ and deg $Q_1=\text{deg }P_1+1$. This proves the result.
\end{proof}

\begin{point}
\label{Thm:5.2}
Eisenbud operators \\
\normalfont
Let $Q$ be a local ring and $\textbf{f}=f_1,\ldots ,f_c$ be a $Q$-regular sequence. Set $A=Q/(\textbf{f}).$ The Eisenbud operators (see \cite{Eisenbud}) are constructed as follows:

Let $\mathbb{F}:\cdots \mathbb{F}_{i+2}\xrightarrow{\partial}\mathbb{F}_{i+1}\xrightarrow{\partial}\mathbb{F}_i\ldots$ be a complex of free $A$-modules.\\
(a) Choose a sequence of free $Q$-modules $\wt{F_i}$ and maps $\wt{\partial}$ between them $$\wt{\mathbb{F}}:\cdots \wt{\mathbb{F}}_{i+2}\xrightarrow{\wt{\partial}}\wt{\mathbb{F}}_{i+1}\xrightarrow{\wt{\partial}}\wt{\mathbb{F}}_i\ldots$$ such that $\mathbb{F}=\wt{\mathbb{F}}\otimes A$.\\
(b) Since $\wt{\partial}^2\equiv 0$ modulo$(f_1,\ldots ,f_c),$ we may write  $\wt{\partial}^2=\sum_{j=1}^{c} f_j\wt{t_j}$ where $\wt{t_j}:\wt{\mathbb{F}}_i\to \wt{\mathbb{F}}_{i-2}$ are linear maps for all $i.$\\
(c) Define for $j=1,\ldots ,c$ the map $t_j=t_j(Q,\textbf{f},\mathbb{F}): \mathbb{F}\to \mathbb{F}$ by $t_j=\wt{t_j}\otimes A.$

The operators $t_1,\ldots ,t_c$ are called Eisenbud operators associated to $\textbf{f}.$ It can be shown that
\begin{enumerate}
    \item $t_i$ are uniquely determined up to homotopy.
    \item $t_i,t_j$ commute up to homotopy.
\end{enumerate}
\end{point}

\begin{point}
\label{Thm:5.3}
    \normalfont

Let $R = A[\chi_1, \dots, \chi_c]$ be a polynomial ring with variables $\chi_1, \dots, \chi_c$, considered in degree $2$. 
If $M$ and $N$ are finitely generated $A$-modules, then one can define a structure of a graded $R$-module on $\text{Ext}_A^*(M, N)$ in such a way that for each $1 \leq j \leq c$, the multiplication by $\chi_j$ map:
$$\chi_j : \text{Ext}_A^*(M, N) \to \text{Ext}_A^{*+2}(M, N)$$
is the map induced in homology by $\text{Hom}_A(t_j, N)$, where $t_j = t_j(Q, \textbf{f}, \mathbb{F})$, with $\mathbb{F}$ being a free resolution of $M$. Let $S = k[\chi_1, \dots, \chi_c]$. Then $\text{Ext}_A^*(M, k)$ has an induced structure of an $S$-module.

\end{point}

\begin{point}
\label{Thm:5.4}
\normalfont
    In (\cite{Gulliksen}, 3.1), Gulliksen proved that if $\text{projdim}_Q(M)$ is finite then $\operatorname{Ext}_A^*(M,N)$ is finitely generated $R$-module for all  $A$-modules $N$. In (\cite{Avramov}, 3.10), Avramov proved a converse for $N=k.$ Note that if $\text{projdim}_Q(M)$ is finite then \, $\operatorname{ Ext}_A^*(M,k)$ is finitely generated graded $S$-module. Note that the  $\text{dim}_S\operatorname{Ext}_A^*(M,k)$ is equal to the complexity of $M$ over $A,$ where the complexity of $M$ over $A$ is defined as: $$\textrm{cx}_A(M)=\inf \left\{b\in \N \quad \vline \quad {\limsup_{n\to \infty}} \,\, \frac{\beta_n^A(M)}{n^{b-1}}< \infty \right\}.$$
\end{point}
Note that $\text{cx}_A(M)=0$ implies that the module has finite projective dimension. If a module is maximal Cohen-Macaulay with $\text{cx}_A(M)=0$, then $M$ is free. 

\emph{From now on, when we say maximal Cohen-Macaulay, we mean non-free maximal Cohen-Macaulay.}

We will repeatedly use the following result. Its proof is similar to Theorem 3.7 of \cite{cxM}. We include a proof for the convenience of the reader.

\begin{proposition}
\label{cxM}
  Let $(A,\m,k)$ be a complete complete intersection ring with infinite residue field and $M$ a maximal Cohen-Macaulay $A$-module of complexity $\geq 2$. Set $M_i=\text{Syz}^A_i(M)$ the $i$-th sygyzy of $M$ and $R=A[\chi_1, \dots, \chi_c]$ (see \ref{Thm:5.3}). Fix $s\geq 1.$ Then there exist $\alpha\in R_2$ (linear in $\chi_i$) and $n_0\in \N$ (depending on $s$) such that $\alpha$ induces surjective maps $M_{i+2}\xrightarrow{\alpha_{i,M}} M_{i}$ for all $i\geq n_0$ and  the following properties hold:
    \begin{enumerate}
        \item The maps $\Tor^A_j(\alpha_{n_0,M}, A/\m^n): \Tor^A_j(M_{n_0+2},A/\m^n)\to \Tor^A_j(M_{n_0},A/\m^n)$ are surjective for all $j\geq 1$ and $1\leq n\leq s.$
        \item Set $K_i=\text{ker}(\alpha_{i,M})$ for each $i\geq n_0.$ Then we have a short exact sequence $0\to K_i\to M_{i+2}\xrightarrow{\alpha_{i,M}} M_{i}\to 0$ with $\operatorname{cx}_A(K_i)=\cx_A(M)-1$ for all $i\geq n_0.$
        \item  $K_{i+1}\cong \text{Syz}^A_1(K_i)$ for all $i\geq n_0.$
    \end{enumerate} 
\end{proposition} 

\begin{proof}
Let $A=Q/\textbf{f}$, where $Q$ is a regular local ring and $\textbf{f}=f_1,\ldots f_c$ be a $Q$-regular sequence. Let $\mathbb{F}$ be the minimal free resolution of $M$ as an $A$-module and let $t_1,\ldots, t_c : \mathbb{F}(+2)\to  \mathbb{F}$
be Eisenbud operators corresponding to $\textbf{f}$ and $\mathbb{F}$ (see \ref{Thm:5.2}). Let $E_A(k)$ be the injective hull of $A/\m=k.$ Set $R=A[\chi_1, \dots, \chi_c]$ as in \ref{Thm:5.3} and $E_i(M)=\operatorname{Ext}^*_A(M,(A/\m^i)^{\vee})$ for all $i\geq 1$, where $(A/\m^i)^{\vee}=\Hom(A/\m^i,E_A(k)).$ Note that $E_i(M)$ is a finitely generated graded $R$-module for each $i\geq 1$ (see \ref{Thm:5.4}). Set $E=\bigoplus_{i=1}^{s}E_i(M).$ Since the residue field of $A$ is infinite,  we may choose a homogenous element $\alpha$ (linear in $\chi_i$) in $R_2$ such that $(0:_{E}\alpha)$ has finite length. It follows that we have the following injections 

\begin{equation}\label{g}
    \Ext_A^i(M,(A/\m^n)^{\vee})\xrightarrow{\alpha} \Ext_A^{i+2}(M,(A/\m^n)^{\vee}) \text{ for all $1\leq n\leq s$}
\end{equation} 
for all $i\gg0$ (say $i\geq n_0$). 

Let $\alpha=a_1\chi_1+\ldots +a_c\chi_c$. Note that $\alpha'=a_1t_1+\ldots +a_ct_c : \mathbb{F}(+2)\to \mathbb{F}$ is a chain map and the multiplication map in \eqref{g} is the map induced in homology by $\displaystyle \bigoplus_{n=1}^{s}\Hom(\alpha',(A/\m^n)^{\vee}))$ (see \ref{Thm:5.3}).

By dualizing with respect to the injective hull of $A/\mathfrak{m}$, we obtain surjections where the maps are those induced by $\alpha'$.
$$(\Ext_A^{i+2}(M,(A/\m^n)^{\vee}))^{\vee}\xrightarrow{} (\Ext_A^{i}(M,(A/\m^n)^{\vee}))^{\vee}.$$
We also have the natural isomorphism $$\Hom(\mathbb{F}\otimes A/\m^n,E_A(k))\cong \Hom(\mathbb{F},(A/\m^n)^{\vee})).$$ By Theorem 3.2.13 of \cite{Bruns}, this induces surjective maps
$$ \Tor^A_i(\alpha',A/\m^n): \Tor^A_{i+2}(M,A/\m^n)\xrightarrow{} \Tor^A_{i}(M,A/\m^n) \text{ for all $1\leq n\leq s$}$$  for all $i\geq n_0$. 

Since $\Tor^A_i(\alpha',A/\m)$ is surjective for all $i\geq n_0$, by Nakayama's lemma, we have surjections $\mathbb{F}_{i+2}\xrightarrow{\alpha_{i,M}} \mathbb{F}_i$ for all $i\geq n_0,$ where $\alpha_{i,M}$ is the $i$th component map of the chain map $\alpha'.$ This implies that for all $i\geq n_0$, we have the following short exact sequence:
\begin{equation}\label{f}
   0 \to K_i\to M_{i+2}\xrightarrow{\alpha_{i,M}}M_i\to 0,
\end{equation}
 where $K_i=\text{ker}(\alpha_{i,M})$. Since $\Tor_i(M,-)\cong\Tor_{i-1}(M_1,-)$ for all $i\geq 1$, we have surjections $\Tor^A_{i}(M_{n_0+2},A/\m)\xrightarrow{} \Tor^A_{i}(M_{n_0},A/\m)$ for all $i\geq 0.$ It follows that for all $i\geq 0$, we have the following exact sequence: $$0\to \Tor^A_{i}(K_{n_0},A/\m)\to \Tor^A_{i}(M_{n_0+2},A/\m)\xrightarrow{\Tor^A_i(\alpha_{n_0,M},A/\m)} \Tor^A_{i}(M_{n_0},A/\m)\to 0.$$ This implies that $\text{cx}_A(K_{i})=\text{cx}_A(M_{i})-1$ for all $i\geq n_0$. This completes the proofs of parts $(1)$ and $(2).$

Now we have the following commutative diagram with exact rows and columns for all $i\geq n_0$: 
\begin{equation*}
    \begin{tikzcd}
     0 \arrow{r}{} & \overline{\mathbb{F}}_{i+2} \arrow{r}{} \arrow{d}{}  &  \mathbb{F}_{i+4} \arrow{r}{\alpha_{i+2,M}} \arrow{d}{\partial_{i+4}} & \mathbb{F}_{i+2} \arrow{r}{} \arrow{d}{\partial_{i+2}} & 0 \\
   0 \arrow{r}{} & \overline{\mathbb{F}}_{i+1} \arrow{r}{} \arrow{d}{\overline{\partial_{i+1}}}  &  \mathbb{F}_{i+3} \arrow{r}{\alpha_{i+1,M}} \arrow{d}{\partial_{i+3}} & \mathbb{F}_{i+1} \arrow{r}{} \arrow{d}{\partial_{i+1}} & 0 \\  0 \arrow{r}{} & \overline{\mathbb{F}}_{i} \arrow{r}{} &
        \mathbb{F}_{i+2} \arrow{r}{\alpha_{i,M}} & \mathbb{F}_{i} \arrow{r}{} & 0,
    \end{tikzcd}
\end{equation*} 
where $\overline{\mathbb{F}}_{i}$ are the kernel of the corresponding maps and $\overline{\partial_{i+1}}$ is the restriction map of $\partial_{i+3}$ for all $i\geq 0$. Note that here each row is a split exact sequence and $M_i$ is the cokernel of $\mathbb{F}_{i+1}\xrightarrow{\partial} \mathbb{F}_i$ for each $i\geq 0$. Applying the Snake Lemma to the bottom two rows of the commutative diagram above and using \eqref{f}, we obtain $$0\to \text{coker}(\overline{\partial})\to M_{i+2}\xrightarrow{\alpha_{i,M}} M_i\to 0$$ for all $i\geq n_0.$ This implies that $K_i=\text{coker}(\overline{\partial_{i+1}})$  for all $i\geq n_0.$ This proves part (3).

\end{proof}

We end this section by recalling the definition of the reduction number and establishing a convention.

 Let $(A,\m)$ be a local ring and $I$ an ideal of $A.$ A reduction of $I$ is an ideal $J \subseteq I$ such that $JI^n =I^{n+1}$ for some $n \in N$. If $J$ is a reduction of $I$, the reduction number of $I$ with respect to $J$ is defined as: $$\text{rd}_J(I) =\text{min}\{n\, |\, JI^n = I^{n+1}\}.$$ A reduction is minimal if it is minimal with respect to inclusion. The reduction number of $I$ is defined as $\text{rd}(I) = \text{min}\{\text{rd}_J(I) \, |\, J \text{ is a minimal reduction of }  I\}$. For $I=\m$, we write red$(A)$ instead of $\text{rd}(\m).$

 \begin{point}\label{conv}
     \textbf{Convention:}
\normalfont
From now on, all statements that use the notation of Proposition \ref{cxM} will assume $s=\operatorname{red}(A)$.
 \end{point}

\section{Degree of Quasi-polynomials}
Let $\mathbb{F}$ be a minimal resolution of $M$ over A. Set $M_i=\text{Syz}_i^A(M)$ for all $i\geq 1$ and $M_0=M.$ In this section, we discuss the behavior of Hilbert coefficients of $M_i$ for all $i\gg0$.

The following result was proved by the first author (see \cite{Quasipolynomial}, Corollary 2.14). We provide an alternative proof here.
\begin{proposition}
\label{e_0}
    Let $(A,\m)$ be a complete intersection ring and $I$ a $\m$-primary ideal. Let $M$ be a maximal Cohen-Macaulay module. Then the function $i\mapsto e^I_0(M_i)$ is of quasi-polynomial type with period $2$ and degree equal to $\operatorname{cx}_A(M)-1$. Furthermore, the associated polynomials have the same degree equal to $\text{cx}_A(M)-1.$
\end{proposition}
\begin{proof}
   By doing a faithful flat base change, Hilbert series, complexity, and complete intersection property of the ring do not change; we may assume $A$ is complete and the residue field of $A$ is uncountable (see \ref{basechange}). We will prove the theorem by induction on $\operatorname{cx}_A(M).$ Assume $\operatorname{cx}_A(M)=1.$ We may assume $M$  has no free summand. Since $M$ is a maximal Cohen-Macaulay module, $M$ has periodic resolution of period $2$, $i.e.$ for all $i\geq 0$, $\text{Syz}^A_{2i+j}(M)\cong\text{Syz}^A_j(M)$ for $j=0,1$ (see \cite{avramov1998infinite}, Theorem 9.2.1 and \cite{Eisenbud}, Theorem 6.1). Therefore $e^I_0(M_{2i})=e^I_0(M)\neq 0$ and $e^I_0(M_{2i+1})=e^I_0(M_1)\neq 0.$ This implies that both the function $i\mapsto e^I_0(M_{2i})$ and $i\mapsto e^I_0(M_{2i+1})$ are of polynomial type with degree equal to $\text{cx}_A(M)-1=0.$ Hence the statement is true for $\operatorname{cx}_A(M)=1$. 
    
Let us assume the statement is true for $\operatorname{cx}_A(M)=r-1.$ Let $\operatorname{cx}_A(M)=r\geq 2.$ Fix $s=1$ and by Proposition \ref{cxM}, we have a short exact sequence: $$0\to K_i \to M_{i+2}\xrightarrow{\alpha_{i,M}} M_i\to 0$$ for all $i\gg0$ (say $i\geq n_0$) with $\operatorname{cx}(K_{n_0})=\operatorname{cx}_A(M)-1$. Note that $K_j\cong \text{Syz}^A_{j-n_0}(K_{n_0}).$ Since multiplicity is additive on short exact sequences we get  $e^I_0(M_{i+2})-e^I_0(M_{i})=e^I_0(K_{i})$ for all $i\gg0$. By the induction hypothesis, the function $i\mapsto e^I_0(K_{i})$ is of quasi-polynomial type with period $2$ and the associated polynomials have the same degree equal to $\cx(M)-2$. Therefore, the function $i\mapsto e^I_0(M_{i})$ is of quasi-polynomial type with period $2$ and the associated polynomials have the same degree equal to  $\operatorname{cx}_A(M)-1$ (see Corollary \ref{deg}).
\end{proof}

\begin{lemma}
\label{exact}
     Let $(Q,\n)$ is regular local ring and $f_1,\ldots ,f_c \in \n^2$ be $Q$-regular sequence.  Set $(A,\m)=(Q/(f_1,\ldots ,f_c), \overline{\n})$ be a complete intersection ring of dimension $1$ with an uncountable residue field. Let $M$ be a maximal Cohen-Macaulay $A$-module. Then we have the following exact sequence: $$0\to L(K_i)\to L(M_{i+2})\xrightarrow{L(\alpha_{i,M})} L(M_i)\to 0$$ for all $i\gg0$ (for definition of $K_i$ and $\alpha_{i,M}$, see Proposition \ref{cxM} and Convention \ref{conv}).
\end{lemma}
\begin{proof}
     Let $x\in \m$ be a $M_i\oplus A$-superficial element for all $i\geq 0$ (see \cite{Quasipolynomial}, Lemma 2.2). By (\cite{Quasipolynomial}, Lemma 4.1(3)), the map  $\overline{\delta}^n_x:\operatorname{Tor}^A_1(M_i, A/\m^n)\to  \operatorname{Tor}^A_1(M_i, A/\m^{n+1})$ is an isomorphism for all  $n\geq \operatorname{red}(A),$ where $\overline{\delta}^n_x$ is the map induced by $\delta^n_x:A/\m^n\to A/\m^{n+1}$, defined by $\delta^n_x(a+\m^n)=ax+\m^{n+1}$. Since $s=\text{red}(A)$ and by Proposition \ref{cxM}, for all $1\leq r\leq \text{red}(A)$, we have the following surjective maps
     \begin{equation}
     \label{equation}
         \operatorname{Tor}^A_1(M_{i+2},A/\m^r)\xrightarrow{\Tor^A_1(\alpha_{i,M}, A/\m^r)} \operatorname{Tor}^A_1(M_i,A/\m^r)
     \end{equation}
      for all $i\gg0$ (say $i\geq {n_0}$). Let  $n=\operatorname{red}(A)$. For $i\geq n_0$, we have the following commutative diagram:

    \begin{equation*}
    \begin{tikzcd}
        \mathbb{F}_{i+2}\otimes A/\m^n \arrow{rr}{\alpha_{i,M}\otimes 1_n} \arrow{d}{\overline{\delta}^n_{x}} &&  \mathbb{F}_{i}\otimes A/\m^n \arrow{d}{\overline{\delta}^{n}_x}  \\
         \mathbb{F}_{i+2}\otimes A/\m^{n+1} \arrow{rr}{\alpha_{i,M}\otimes 1_{n+1}} &&  \mathbb{F}_{i}\otimes A/\m^{n+1}.
    \end{tikzcd}
\end{equation*}
This induces the following commutative diagram:
\begin{equation*}
    \begin{tikzcd}
        \operatorname{Tor}^A_1(M_{i+2},A/\m^n) \arrow{rrr}{\Tor^A_1(\alpha_{i,M}, A/\m^n)} \arrow{d}{\overline{\delta}^n_x} &&& \operatorname{Tor}^A_1(M_i,A/\m^n) \arrow{r}{} \arrow{d}{\overline{\delta}^n_x} & 0 \\
        \operatorname{Tor}^A_1(M_{i+2},A/\m^{n+1}) \arrow{rrr}{\Tor^A_1(\alpha_{i,M}, A/\m^{n+1})} &&& \operatorname{Tor}^A_1(M_i,A/\m^{n+1}) \arrow[dashed]{r}{} & 0,
    \end{tikzcd}
\end{equation*}
 since $\overline{\delta}^n_x$ is an isomorphism and by \eqref{equation}, we get a surjective map $$\operatorname{Tor}^A_1(M_{i+2},A/\m^{n+1}) \to  \operatorname{Tor}^A_1(M_i,A/\m^{n+1}).$$ Now replacing $n$ by $n+1, \, n+2,\ldots$, we obtain the following exact sequence:
 \begin{equation}\label{imp}
     \operatorname{Tor}^A_1(M_{i+2},A/\m^{n}) \to  \operatorname{Tor}^A_1(M_i,A/\m^{n})\to 0,
 \end{equation}
  for all $i\geq n_0$ and $n\geq 1$. 

 Now for all $i\geq n_0$, we also have the following exact sequence (see Proposition \ref{cxM}): $$0\to K_i\to M_{i+2}\xrightarrow{\alpha_{i,M}}M_i\to 0.$$ By tensoring with $A/\m^n$, we obtain the following long exact sequence of homology: $$\Tor^A_1(M_{i+2},A/\m^n)\to \Tor^A_1(M_{i},A/\m^n)\to K_i\otimes A/\m^n\to M_{i+2}\otimes A/\m^n\xrightarrow{\alpha_{i,M}\otimes 1_n} M_i\otimes A/\m^n\to 0$$ for all $i\geq n_0.$ By using \eqref{imp}, we get the exact sequence: $$0\to K_i\otimes A/\m^n\to M_{i+2}\otimes A/\m^n\xrightarrow{\alpha_{i,M}\otimes 1_n} M_i\otimes A/\m^n\to 0,$$  for all $i\geq n_0$ and $n\geq 1$. The result now follows by taking the direct sum over $n\geq 1$.

\end{proof}

\begin{theorem}
\label{e1}
     Let $(A,\m)$ be a complete intersection ring of codimension $c\geq 2$ and dimension $d\geq 1$. Let $M$ be a maximal Cohen-Macaulay $A$-module. Set $M_i=\operatorname{syz}^A_i(M)$ for all $i\geq 0.$ Then the function $i\mapsto e_1(M_i)$ is of quasi-polynomial type with period $2$ and  the associated polynomials have the same degree equal to ${\cx}_A(M)-1.$
\end{theorem}
\begin{proof}
We may assume $A$ is complete and the residue field of $A$ is uncountable (see \ref{basechange}). Let $x_1,\ldots ,x_{d}$ be $(M_i\oplus A)$-superficial sequence for all $i\geq 0$ (see \cite{Quasipolynomial}, Lemma 2.2). By \ref{hcs}, we have the dimension of $M_i/(x_1,\ldots ,x_{d-1})M_i$ is one and $e_1(M_i)=e_1(M_i/(x_1,\ldots ,x_{d-1})M_i)$  for all $i\geq 0$. So it suffices to show for $d=1.$

Since the residue field of $A$ is infinite, we can choose $x\in \m\setminus \m^2$ as a $(M\oplus A)$-superficial element. By \ref{hcs}, $e_1(M)\geq e_0(M)-\mu(M)\geq 0$. If $e_1(M)=0$ then $e_0(M)=\mu(M)$. We also have $e_0(M)=e_0(M/xM)=\ell_A(M/xM)$ and $\ell_A(M/\m M)=\mu(M)=\mu(M/xM)$ (see \ref{hcs}). It follows that $M/xM\cong M/\m M\cong k^{\mu(M)}$, where $k$ is the residue field of $A.$ As $x\in \m \setminus \m^2$, we can choose a minimal set of generator of $\m$ such that $x$ is a part of it. 

    Assume $\operatorname{cx}_A(M)=1$. We may assume $M$  has no free summand. Since $M$ is a maximal Cohen-Macaulay module, it has periodic resolution of period $2$, $i.e.$ for all $i\geq 0,$ $\text{Syz}^A_{2i+j}(M)\cong\text{Syz}^A_j(M)$ for $j=0,1$ (see \cite{avramov1998infinite}, Theorem 9.2.1 and \cite{Eisenbud}, Theorem 6.1). Therefore,  $e_1(M_{2i})=e_1(M)$ and $e_1(M_{2i+1})=e_1(M_1).$ Suppose $e_1(M)=0$. Then $M/xM\cong k^{\mu(M)}$. Let $\mathbb{F} \to M$ be the minimal free resolution of $M$. Then $\mathbb{F}/x\mathbb{F}\to M/xM$ is a minimal free resolution of $M/xM\cong k^{\mu(M)}$. Consequently, $k$ has a periodic resolution of period $2$. Therefore, by (\cite{Gulliksen}, Theorem 2.3), $A/xA$ is a hypersurface ring. As $x\in \m \setminus \m^2$, codimension of  $A=\mu(\m)-\text{dim} A=\mu(\m/x)-\text{dim} A/x\leq 1$ and hence $A$ is a hypersurface ring. This contradicts the hypothesis that the codimension of $A$ is at least $2$. Similarly, we can show that $e_1(M_1)\neq 0.$ This implies that both the function $i\mapsto e_1(M_{2i})$ and $i\mapsto e_1(M_{2i+1})$ are of polynomial type with degree equal to $\text{cx}_A(M)-1=0.$ Hence, the statement is true for $\operatorname{cx}_A(M)=1$.

     Let us assume the statement is true for $\operatorname{cx}_A(M)=r-1.$ Let $\operatorname{cx}_A(M)=r\geq 2.$ Fix $s=\text{red}(A)$ and by Proposition \ref{cxM}, we have short exact sequence: $$0\to K_i \to M_{i+2}\xrightarrow{\alpha_{i,M}} M_i\to 0$$ for all $i\gg0$ (say $i\geq n_0$) with $\operatorname{cx}(K_{n_0})=\operatorname{cx}_A(M)-1$ and by Lemma \ref{exact}, we have the following short exact sequence: $$0\to L(K_i)\to L(M_{i+2})\xrightarrow{L(\alpha_{i,M})} L(M_i)\to 0$$ for all $i\gg0.$ By \ref{samuel}, we have $h_{K_i}(z)=h_{M_{i+2}}(z)-h_{M_i}(z)$. This implies  $e_1(M_{i+2})-e_1(M_{i})=e_1(K_{i})$ for all $i\gg0.$ By the induction hypothesis, the function $i\mapsto e_1(K_{i})$ is of quasi-polynomial type with period $2$ and degree equal to $\cx(M)-2$. Therefore, the function $i\mapsto e_1(M_{i})$ is of quasi-polynomial type with period $2$ and the associated polynomials have the same degree equal to $\text{cx}_A(M)-1$ (see Corollary \ref{deg}).
\end{proof}

We now present examples showing that the equality of the leading coefficients of $P_0$ and $P_1$ in Proposition \ref{e_0} and Theorem \ref{e1} (equivalently, in Question \ref{que}(b) for $j=0,1$) does not hold in general.

\begin{example}\label{example1}
\normalfont
\begin{enumerate}
    \item Let $(A,\m)=(k[|x,y|]/(x^2y^2),(\overline{x},\overline{y}))$, where $k$ is a field and $M=A/(\overline{x^2})$. Note that $M$ is a maximal Cohen-Macaulay $A$-module with $\text{cx}_A(M)=1.$ It is easy to check that $e_0(M_i)=2$ and $e_1(M_i)=1$ for all $i\geq 0.$ It follows that there exist $A$ and $M$ such that Question \ref{que}(a,b) has a positive answer with the associated polynomials $P_0$ and $P_1$ having the same leading coefficient for $j=0,1$.
    \item Let $M$ be a maximal Cohen-Macaulay module with $e_0(M)=\mu(M)$ over a hypersurface ring with $e_0(A)>2$. Then, $e_0(M_{2i})\neq e_0(M_{2i+1})$ for all $i\geq 0$ (see \cite{HCMCMPart1}, Theorem 2, Remark 4.10).
    \item Let $(A,\m)=(k[|x,y|]/(xy^2),(\overline{x},\overline{y}))$ and $M=A/(\overline{x})$. Note that $M$ is a maximal Cohen-Macaulay $A$-module with $\cx_A(M)=1$. It is easy to check that $e_1(M_{2i})=0$ and $e_1(M_{2i+1})=1$ for all $i\geq 0.$
\end{enumerate}
Examples (2) and (3) provide a negative answer to Question \ref{que}(b) for $j=0,1$ respectively.
\end{example}

\section{Leading coefficients of quasi-polynomials}
In the previous section, we established that for a maximal Cohen-Macaulay module $M$ over a ring $A$, the functions $i\mapsto e_0(M_i)$, and $i\mapsto e_1(M_i)$ are of quasi-polynomial type with period $2$ and the associated polynomials have the same degree equal to $\operatorname{cx}_A(M)-1$. Due to Theorem $4.1$ in \cite{Avramov}, the function $i\mapsto \mu(M_i)$ is of quasi-polynomial type with period $2$ and the associated polynomials have the same degree equal to $\operatorname{cx}_A(M)-1$ and leading coefficient. In this section, we will explore the relationship between their leading coefficients.
\begin{lemma}

\label{cxM2}
    Let $(A,\m)$ be a complete intersection ring of codimension $c\geq 2$ and dimension $d\geq 1$. Let $M$ be a maximal Cohen-Macaulay $A$-module. Then for $j=0,1$ 
    \begin{equation}\label{2.1}
        \lim_{i\to \infty}\dfrac{e_1(M_{2i+j})}{i^{\operatorname{cx}_A(M)-1}}\geq \lim_{i\to \infty}\dfrac{e_0(M_{2i+j})}{i^{\operatorname{cx}_A(M)-1}}-\lim_{i\to \infty}\dfrac{\mu(M_{2i+j})}{i^{\operatorname{cx}_A(M)-1}},
    \end{equation}
  where $\mu(M_i)=\text{number of minimal generators of } M_i=\beta_i(M).$ 
\end{lemma}
\begin{proof}
    We may assume $A$ is complete and the residue field is uncountable (see \ref{basechange}). By (\cite{Avramov}, Theorem 4.1), Theorem \ref{e_0} and Theorem \ref{e1}, the function $i\mapsto \mu(M_i)$, $i\mapsto e_0(M_i)$ and $i\mapsto e_1(M_i)$ are of quasi-polynomial type of period $2$ and the associated polynomials have the same degree equal to $\operatorname{cx}_A(M)-1$. Therefore, for $i$ large enough, let 
\begin{equation}\label{3.1}
\begin{split}
    \mu(M_i)=  \left\{
\begin{array}{ll}
      \gamma i^{\operatorname{cx}_A(M)-1}+ \text{lower degree terms} & i \text{ even}; \\
      \gamma i^{\operatorname{cx}_A(M)-1}+ \text{lower degree terms} & i \text{ odd},\\
\end{array} 
\right. \\
e_0(M_i)=  \left\{
\begin{array}{ll}
      \beta' i^{\operatorname{cx}_A(M)-1}+\text{lower degree terms} & i \text{ even}; \\
      \beta i^{\operatorname{cx}_A(M)-1}+\text{lower degree terms} & i \text{ odd},\\
\end{array} 
\right. \\
e_1(M_i)=  \left\{
\begin{array}{ll}
      ai^{\operatorname{cx}_A(M)-1}+\text{lower degree terms} & i \text{ even}; \\
      bi^{\operatorname{cx}_A(M)-1}+\text{lower degree terms} & i \text{ odd},\\
\end{array} 
\right.
\end{split}
\end{equation}
where $\beta', \beta, \gamma, a, b\in \mathbb{Q}.$
We also have $e_1(M_i)\geq e_0(M_i)-\mu(M_i)$ for all $i\geq 0$ (see \ref{hcs}). This implies 
$$\lim_{i\to \infty}\dfrac{e_1(M_{2i+j})}{i^{\operatorname{cx}_A(M)-1}}\geq \lim_{i\to \infty}\dfrac{e_0(M_{2i+j})}{i^{\operatorname{cx}_A(M)-1}}-\lim_{i\to \infty}\dfrac{\mu(M_{2i+j})}{i^{\operatorname{cx}_A(M)-1}},$$ for $j=0,1.$ This proves the result.

\end{proof}

\begin{lemma}\label{cxM2.1}
    (with the same hypothesis as in Lemma \ref{cxM2}) Let $\operatorname{cx}_A(M)=2$ and assume equality holds in (\ref{2.1}). Then $G(K_i)$ is Cohen-Macaulay for all $i\gg0$ (for definition of $K_i$, see Proposition \ref{cxM} and Convention \ref{conv}). 
\end{lemma}
\begin{proof}
     We may assume $A$ is complete and the residue field is uncountable (see \ref{basechange}). Let $x$ be $M_i\oplus A$-superficial element for all $i\geq 0$ (see \cite{Quasipolynomial}, Lemma 2.2). By \ref{superficial}, $x$ is both $M$ and $A$-regular. We have the following short exact sequence: $$0\to \text{Syz}^A_1(M)\to A^{\mu(M)}\to M\to 0.$$ By tensoring with $A/x$, we get $$\ldots \to \Tor^A_1(M, A/x)\to \text{Syz}^A_1(M)/x\text{Syz}^A_1(M)\to (A/x)^{\mu(M)}\to M/xM\to 0.$$ Due to Lemma $2$(iii) of (\cite{Matsumura}, page 140), $$\Tor^A_1(M,A/x)\cong \Tor^{A/x}_1(M/xM,A/x)=0.$$ This implies that if $M_i$ is the $i$-th syzygy of $M$ then $M_i/xM_i$ is the $i$-th syzygy of $M/xM.$ This also implies that $\text{cx}_A(M)=\text{cx}_{A/x}(M/xM).$
     
     Let $x_1,\ldots ,x_{d}$ be $(M_i\oplus A)$-superficial sequence for all $i\geq 0$ (see \cite{Quasipolynomial}, Lemma 2.2,). By \ref{hcs}, we have the dimension of $M_i/(x_1,\ldots ,x_{d-1})M_i$ is one, $e_0(M_i)=e_0(M_i/(x_1,\ldots ,x_{d-1})M_i)$, $e_1(M_i)=e_1(M_i/(x_1,\ldots ,x_{d-1})M_i)$, and $\mu(M_i)=\mu(M_i/(x_1,\ldots ,x_{d-1})M_i)$ for all $i\geq 0$. Therefore, we may assume the dimension of $A$ is one. 

      By the Lemma \ref{exact}, we obtain the short exact sequence: $$0\to L(K_{i})\to L(M_{i+2})\xrightarrow{L(\alpha_{i,M})} L(M_{i})\to 0$$ for all sufficiently large $i$ and this implies the following equalities:
\begin{equation*}
    \begin{split}
        \mu(M_{i+2})-\mu(M_{i})= & \, \mu(K_{i}) \\  e_0(M_{i+2})-e_0(M_{i})= & \, e_0(K_{i}) \\ e_1(M_{i+2})-e_1(M_{i})= & \,  e_1(K_{i}).
    \end{split}
\end{equation*}

Note that from (\ref{3.1}), we have 

\begin{equation*}
\begin{split}
    \mu(M_i)=  \left\{
\begin{array}{ll}
      \gamma i+ k_1 & i \text{ even}; \\
      \gamma i+ k_2 & i \text{ odd},\\
\end{array} 
\right. \\
e_0(M_i)=  \left\{
\begin{array}{ll}
      \beta' i+k_3 & i \text{ even}; \\
      \beta i+k_4 & i \text{ odd},\\
\end{array} 
\right. \\
e_1(M_i)=  \left\{
\begin{array}{ll}
      ai+k_5 & i \text{ even}; \\
      bi+k_6 & i \text{ odd},\\
\end{array} 
\right.
\end{split}
\end{equation*}
where $k_1,\ldots ,k_6\in \mathbb{Q}.$
Assume that the equality holds for $j=0$, i.e.

$$\lim_{i\to \infty}\dfrac{e_1(M_{2i})}{i}= \lim_{i\to \infty}\dfrac{e_0(M_{2i})}{i}-\lim_{i\to \infty}\dfrac{\mu(M_{2i})}{i}.$$ 
Thus $a=\beta'-\gamma$. This implies 
\begin{equation*}
    \begin{split}
        e_0(K_{2i}) - \mu(K_{2i})= & e_0(M_{2i+2})-e_0(M_{2i})-\mu(M_{2i+2})+\mu(M_{2i})\\ = & e_0(M_{2i+2})-\mu(M_{2i+2}) -(e_0(M_{2i})-\mu(M_{2i}))\\= & (\beta'-\gamma)(2i+2)+(k_3-k_1)-[(\beta'-\gamma)(2i)+(k_3-k_1)]\\= & 2a\\=&e_1(K_{2i}),
    \end{split}
\end{equation*}
 for all $i\gg0$. Therefore, $G(K_{2i})$ has minimal multiplicity, and hence $G(K_{2i})$ is Cohen-Macaulay (see \ref{hcs}). Similarly, by assuming the equality holds for $j = 1$, we can show that $G(K_{2i+1})$ is Cohen-Macaulay for all $i \gg 0$. This proves the result.
\end{proof}

We now give an example where the equality for Lemma \ref{cxM2.1} holds in the case of $\text{cx}(M)=2.$ Recall that a $\N$-graded $k$-algebra $R$, where $k$ is a field, is said to be Koszul if the residue field of $R$ has a linear resolution.

 Let $A=K[[X_1,\ldots ,X_n]]$, $\m=(X_1,\ldots ,X_n)$ and $\phi:A^n\to A^m$ be a $A$-linear map. Note that $R=G_{\m}(A)\cong K[x_1,\ldots ,x_n]$. Assume image$(\phi)\subseteq \m^sA^m$ but not in $\m^{s+1}A^m$ for some $s\geq 0.$ Let $\phi = (a_{ij} )$, where
$a_{ij} \in A$. By assumption $a_{ij} \in \m^s$. It follows that $\phi=\sum_{j\geq s}\phi_s$, where $\phi_j$ is a matrix with entries homogeneous forms of degree $j$. The initial form of $\phi$ is denoted as $\text{in}(\phi) = \phi_s$.
\begin{example}
\normalfont
    Let $A=k[[X_1,\dots, X_d,Y,Z]]/(Y^2,Z^2)$, where $k$ is an uncountable field. Consequently, $G(A)\cong k[x_1,\dots, x_d,y,z]/(y^2,z^2)$. Note that $A$ is a complete intersection ring of codimension $2$ with multiplicity $4$. This indicates that $\operatorname{cx}_A(k)=2$ (see \cite{avramov1998infinite}, Example 9.2.2) and according to Proposition $8$ of \cite{Conca}, both $A$ and $G(A)$ are Koszul. Therefore, $k$ has a linear resolution as $A$ and $G(A)$-module. 

    Let $$\mathbb{F}:\, \ldots \to \mathbb{F}_i\xrightarrow{\phi_i} \mathbb{F}_{i-1}\xrightarrow{\phi_{i-1}} \ldots \to \mathbb{F}_1\xrightarrow{\phi_1} \mathbb{F}_0\to k\to 0$$ be the minimal free resolution of $k$ as $A$-module. Set $M_i$ be the $i$-th syzygy of $k$. Since $k$ has a linear resolution as $A$ and $G(A)$-module, we can conclude that there exists the following exact sequence of  $G(A)$-modules:

    $$\text{in}(\mathbb{F}:) \ldots \to G(\mathbb{F}_i)(-i)\xrightarrow{\text{in}(\phi_i)} G(\mathbb{F}_{i-1})(-i+1)\xrightarrow{\text{in}(\phi_{i-1})} \ldots \to G(\mathbb{F}_1)(-1)\xrightarrow{\text{in}(\phi_1)} G(\mathbb{F}_0)\to k\to 0$$ and the $i$-th syzygy of $k$ as $G(A)$-module is $G(M_i)$. By applying the depth lemma (see \cite{Bruns}, Proposition 1.2.9), we can conclude that $G(M_i)$ is a maximal Cohen-Macaulay $G(A)$-module for all $i\geq d.$ Note that the reduction number of $\m$ is two. Since $k$ is uncountable we can choose $J=(y_1,\ldots ,y_d)$ to be a minimal reduction of $\m$ such that for all $i\geq d,$ $y^*_1,\ldots ,y^*_d$ is a $G(A)\oplus G(M_i)$-regular sequence. Now we have the following short exact sequences for all $i\geq d$:
    
    $$0\to M_{i+1}\to \mathbb{F}_i\to M_i\to 0.$$
    
    By tensoring with $A/J$, we get $0\to {M_{i+1}}/JM_{i+1}\to {\mathbb{F}_i}/J\mathbb{F}_i\to {M_i}/JM_i\to 0.$ Given the minimality of $\mathbb{F}$, we can conclude that $M_{i+1}/JM_{i+1}\subseteq \m \mathbb{F}_i/J\mathbb{F}_i$. This implies that $\m^2 (M_i/JM_i)=0$ for all $i\geq d+1.$ Consequently,  $G(M_i/JM_i)$ has minimal multiplicity for all $i\geq d+1$ and the same holds true for $G(M_i)$ (see \ref{hcs}(6)). This leads to the conclusion that  $e_1(M_i)=e_0(M_i)-\mu(M_i)$ for all $i\geq d+1$. Thus, we see that equality is maintained in Lemma \ref{cxM2.1}. Furthermore, as derived from the proof of Theorem \ref{e1}, we obtain for all $i\gg 0$, $h_{K_i}(z)=h_{M_{i+2}}(z)-h_{M_i}(z).$ Hence, $G(K_i)$ has minimal multiplicity for all $i\gg0$ and is therefore Cohen-Macaulay.
\end{example}

\section[]
{Asymptotic behavior of $\operatorname{reg} G(M_i)$}
In this section, we aim to demonstrate that if the equality holds in Theorem \ref{cxM2} and $\operatorname{cx}_A(M)=2$, then the set $\{\text{reg }G(M_i)\, |\, i\geq 0\}$ is bounded (refer to Theorem \ref{regMn}). To establish this, we require the following:
\begin{point}
    \normalfont
    Recall that $\R=\bigoplus_{n\geq 0}\m^nt^n$ is the Rees ring with respect to the maximal ideal $\m$. Set $\mathfrak{M}=\m \oplus \R_+$, where $\R_+=\oplus_{n\geq 1}\m^nt^n,$ be the graded maximal ideal of $\R.$ Let $E$ be a graded $\R$-module. We set $H^j(E)$ as the $j$-th graded local cohomology of $E$ with respect to $\mathfrak{M}.$ Define $a(E)=\text{sup}\{n\in \Z \,|\, E_n\neq 0\}.$
\end{point}
\begin{point}
\label{regL}
    \normalfont
    A graded $\R$-module $E$ is considered $*$-Artinian if every descending chain of graded submodules of $E$ terminates. It can be easily proved that if $E$ is $*$-Artinian, then $E_n=0$ for all $n\gg0$. Let $E$ be a finitely generated graded $\R$-module. Then for each $j\geq 0$, $H^j(E)$ is $*$-Artinian (see \cite{Bruns}, Proposition 3.5.4). Set $$a_j(E)=a(H^j(E)).$$ The Castelnuovo-Mumford regularity of $E$ is defined as: $$\text{reg }(E)=\text{max}\{a_j(E)+j\, |\, 0\leq j\leq \text{dim }E \}.$$

From now on, whenever we write $H^j(-)$ and reg (-), we mean $j$-th graded local cohomology of (-) and regularity of $(-)$ with respect to the graded maximal ideal $\mathfrak{M}$ of $\R$ respectively.
\end{point}
\begin{point}
    \normalfont
    Note that we have a natural surjective homogenous homomorphism  $\phi:\R\to G(A)$ with $\phi(\mathfrak{M})=G(A)_+$, where $G(A)_+$ is the irrelevant maximal ideal of $G(A)$. Thus, by the graded independence theorem (see \cite{SBrodman}, 13.1.6), the computation of local cohomology is independent of the choice of the graded base ring.
\end{point}
\begin{proposition}[\cite{Part1}, 4.4]
\label{Art}
    Let $(A,\m)$ be local ring and $M$ be an $A$-module. Set $l=\operatorname{depth}_A(M)$. Then  $H^j(L(M))$ is $*$-Artinian for all $0 \leq j\leq l-1$.
\end{proposition}

The following result concerns the regularity of the associated graded module of $M_i$ for sufficiently large $i$.
\begin{theorem}
\label{regMn}
     Let $(A,\m)$ be a complete intersection ring of codimension $c\geq 2$ and dimension $d\geq 1$. Let $M$ be a maximal Cohen-Macaulay $A$-module. If $\operatorname{cx}_A(M)=2$ and $$\lim_{i\to \infty}\dfrac{e_1(M_{2i+j})}{i}= \lim_{i\to \infty}\dfrac{e_0(M_{2i+j})}{i}-\lim_{i\to \infty}\dfrac{\mu(M_{2i+j})}{i}$$ for $j=0,1$. Then the set $\{\operatorname{reg }G(M_i)\, |\, i\geq 0\}$ is bounded.
\end{theorem}
\begin{proof}
     We may assume $A$ is complete (see \ref{basechange}). Fix $s=\text{red}(A)$ and by Proposition \ref{cxM}, we have the short exact sequence:
     \begin{equation}\label{gap}
         0\to K_{i} \to M_{i+2}\xrightarrow{\alpha_{i,M}} M_{i}\to 0 \text{ for all } i\gg0.
     \end{equation}
     
     \textbf{Claim:} $0\to L(K_{i})\to L(M_{i+2})\xrightarrow{L(\alpha_{i,M})} L(M_{i})\to 0$ is exact for all $i\gg0.$

Note that the claim is true for $d=1$ (see Lemma \ref{exact}). Now, assume $d\geq 2$. From (\ref{gap}), we get the following exact sequence (see \ref{L_1M}): 
 $$L_1(M_{i})\xrightarrow{} L(K_{i})\to L(M_{i+2})\xrightarrow{L(\alpha_{i,M})} L(M_{i})\to 0.$$ Let $g$ be the map from $L_1(M_{i})$ to $L(K_{i})$ in the above exact sequence. Set $U=\text{Image}(g)$. We get the following exact sequence: $$0\to U\to L(K_{i})\to L(M_{i+2})\xrightarrow{L(\alpha_{i,M})} L(M_{i})\to 0.$$ Since $U$ is a homomorphic image of $L_1(M_{i})$, it is a finitely generated graded $\mathcal{R}$-module (see \ref{L_1M}).

Therefore, for all $n\gg0$, we have the following equality 
    \begin{equation*}
        \begin{split}
            \ell_A(U_n)= &\ell_A(K_{i}/\m^{n+1}K_{i})-\ell_A(M_{i+2}/\m^{n+1}M_{i+2})+\ell_A(M_{i}/\m^{n+1}M_{i})\\ = &\sum_{s=0}^{d} (e_s(K_{i})-e_s(M_{i+2})+e_s(M_{i}))\binom{n+d-s}{d-s}.
        \end{split}
    \end{equation*}
    For the second equality, see \ref{samuel}. Since multiplicity behaves well with short exact sequences, and by Lemma \ref{cxM2.1}, we have $e_s(K_{i})-e_s(M_{i+2})+e_s(M_{i})=0$ for $s=0,1$. Therefore 
    
    $$\ell_A(U_n)=(e_2(K_{i})-e_2(M_{i+2})+e_2(M_{i}))\binom{n+d-2}{d-2}+\text{ lower degree terms}.$$
    This says that the dimension of $U$ is at most $d-1$. Suppose $U\neq 0$. By the Lemma \ref{cxM2.1}, $G(K_{i})$ is Cohen-Macaulay for all $i\gg0$ and we also have $\text{Ass}_{\mathcal{R}}(L(K_{i}))=\text{Ass}_{\mathcal{R}}(G(K_{i}))$ (see \cite{TONYT}, Proposition 5.6). Thus Ass$_{\R}(U)\subseteq \text{Ass}_{\R}(L(K_i))=\text{Ass}_{\R}(G(K_i))$, where Ass$_{\R}(-)$ is the set of all associated primes of $(-)$ as an $\R$-module. This implies that dim$_{\R}(U)=\text{dim}_{\R}(G(K_i))=d$. This is a contradiction to the dimension of $U$ being at most $d-1$. Hence $U=0$. This proves the claim.

    Now, the short exact sequence $$0\to L(K_{i})\to L(M_{i+2})\xrightarrow{L(\alpha_{i,M})} L(M_{i})\to 0$$ induces the following long exact sequence of local cohomology:
      \begin{equation}
     \label{4}
         \begin{split}
             0 & \to H^0(L(K_i))\to H^0(L(M_{i+2}))\to H^0(L(M_i)\to \\
             \ldots
             \\ & \to H^{d-1}(L(K_i))\to H^{d-1}(L(M_{i+2}))\to H^{d-1}(L(M_i).
         \end{split}
     \end{equation}
    Since $G(K_i)$ is Cohen-Macaulay of dimension $d$ for all $i\gg0$, we have $H^j(L(K_i))=0$ for $j=0,1,\ldots ,d-1$ (see \cite{Part1}, Proposition 5.2). Therefore, for $j=0,1,\ldots ,d-2$, we have $H^j(L(M_{i+2}))\cong H^j(L(M_{i}))$ for all $i\gg0$. From \eqref{4}, we have the following descending chain
     \begin{equation*}
         H^{d-1}(L(M_{i}))\supseteq H^{d-1}(L(M_{i+2}))\supseteq \ldots
     \end{equation*}
     for all $i\gg0$. Since $H^{d-1}(L(M_i))$ is $*$-Artinian, the descending chain terminates. Therefore, without loss of generality, we may choose $n_0\in \Z$ (even) such that
     \begin{equation}
     \label{2}
         a_j(L(M_{2i+n_0+s}))= a_j(L(M_{n_0+s})) \text{ for all } i\geq 1,\, j=0,1,\ldots ,d-1 \text{ and } s=0,1.
     \end{equation}

    We also have the following short exact sequence of $\R$-modules (see \cite{Part1}, section 5): $$0\to G(M_i)\to L(M_i)\to L(M_i)(-1)\to 0$$ for all $i\geq 0.$ This induces the following long exact sequence: 
     \begin{equation}
     \label{3}
         \begin{split}
             0 & \to H^0(G(M_i))\to H^0(L(M_i))\to H^0(L(M_i)(-1))\to \\
             \ldots
             \\ & \to H^{d-1}(G(M_i))\to H^{d-1}(L(M_i))\to H^{d-1}(L(M_i)(-1)).
         \end{split}
     \end{equation}

From (\ref{3}), we obtain  the following for all $1\leq j\leq d-1$:
\begin{enumerate}
    \item[(a)] $a_0(G(M_i))\leq a_0(L(M_{i}))$,
    \item[(b)] $a_j(G(M_i))\leq \text{max}\{a_{j-1}(L(M_{i}))+1,a_j(L(M_{i})) \}$.
\end{enumerate}

Additionally, it is well known that $a_d(G(M_i))\leq \text{red}(A)-d$ for all $i\geq 0$ (see \cite{NVT3}, 3.2). Hence, by  (\ref{2}), the set $\{\text{reg }G(M_i)\, |\, i\geq 0\}$ is bounded.
\end{proof}
We recover a particular case of Theorem 4.3 and Theorem 5.3 in \cite{Quasipolynomial} as an application.

\begin{corollary}(with the same hypothesis as in \ref{regMn}) The functions $i\mapsto \operatorname{depth}_{G(A)}G(M_{2i})$ and \\ $i\mapsto \operatorname{depth}_{G(A)}G(M_{2i+1})$ are constant for $i\gg0.$
\end{corollary}
\begin{proof}
   By (\ref{2}), we obtain that for all $i\gg0$, $a_j(L(M_{2i}))$ and $a_j(L(M_{2i+1}))$ are constants, for $j=0,1,\ldots ,d-1$. Then Proposition $5.2$ in \cite{Part1} gives the result.
\end{proof}

\providecommand{\bysame}{\leavevmode\hbox to3em{\hrulefill}\thinspace}
\providecommand{\MR}{\relax\ifhmode\unskip\space\fi MR }
\providecommand{\MRhref}[2]{
  \href{http://www.ams.org/mathscinet-getitem?mr=#1}{#2}
}

\end{document}